\documentclass[reqno, 12pt, reqno]{amsart}

\usepackage{amsfonts, amsthm, amsmath, amssymb, bbm, bm, color, enumerate,tikz-cd, eurosym, url}
\usepackage{hyperref}
\usepackage{float}
\usepackage[ruled]{algorithm2e}

\usepackage{mathrsfs}

\usepackage[all]{xy}
\usepackage[margin=2.5 cm]{geometry}

\usepackage{helvet}

\RequirePackage{mathrsfs} \let\mathcal\mathscr

\numberwithin{equation}{section}

\newtheorem{theorem}{Theorem}[section]

\newtheorem{lemma}[theorem]{Lemma}
\newtheorem{proposition}[theorem]{Proposition}
\newtheorem{corollary}[theorem]{Corollary}
\newtheorem{conjecture}[theorem]{Conjecture}

\theoremstyle{definition}

\newtheorem{definition}[theorem]{Definition}
\newtheorem{remark}[theorem]{Remark}

\renewcommand{\phi}{\varphi}
\renewcommand{\rho}{\varrho}

\newcommand{\ZZ}{\mathbb{Z}}

\newcommand{\NN}{\mathbb{N}}
\newcommand{\QQ}{\mathbb{Q}}
\newcommand{\RR}{\mathbb{R}}

\renewcommand{\emptyset}{\varnothing}

\renewcommand{\leq}{\leqslant}
\renewcommand{\geq}{\geqslant}

\newcommand{\x}{\mathbf{x}}
\newcommand{\y}{\mathbf{y}}
\renewcommand{\c}{\mathbf{c}}

\renewcommand{\b}{\mathbf{b}}
\renewcommand{\a}{\mathbf{a}}

\newcommand{\bgamma}{\boldsymbol{\gamma}}

\DeclareMathOperator{\Mod}{mod}
\renewcommand{\bmod}[1]{\,(\Mod{#1})}

\title{On the existence of magic squares of powers}
\author{Nick Rome}
\email{rome@tugraz.at}
\address{TU Graz, Institute of Analysis and Number Theory, Kopernikusgasse 24/II, 8010 Graz, Austria.}

\author{Shuntaro Yamagishi}
\email{shuntaro.yamagishi@ist.ac.at}
\address{IST Austria, Am Campus 1, 3400 Klosterneuburg, Austria.}

\begin{document}
\maketitle

\begin{abstract}
For any $d \geq 2$, we prove that there exists an integer $n_0(d)$ such that there exists an $n \times n$ magic square of $d^\textnormal{th}$ powers for all $n \geq n_0(d)$.
In particular, we establish the existence of an $n \times n$ magic square of squares for all $n \geq 4$, which settles a conjecture of V\'{a}rilly-Alvarado.

All previous approaches had been based on constructive methods and the existence of $n \times n$ magic squares of $d^\textnormal{th}$ powers
had only been known for sparse values of $n$. We prove our result by the Hardy-Littlewood circle method, which in this setting essentially reduces the problem
to finding a sufficient number of disjoint linearly independent  subsets of the columns of the coefficient matrix of the equations defining magic squares. We prove
an optimal (up to a constant) lower bound for this quantity.
%We believe this is the first instance of the circle method being applied to study magic squares.
\end{abstract}

\section{Introduction}
Let $n \geq 1$ be an integer.
A \emph{magic square} is an $n \times n$ grid of distinct positive integers whose columns, rows, and two major diagonals all sum to the same number.
The number to which all rows, columns and diagonals sum is known as the square's \emph{magic constant}.

Magic squares have a long and rich history. Legend has it that the earliest recorded $3 \times 3$ magic square was first observed by Emperor Yu on the shell of a sacred turtle, which emerged from the waters of the Lo River \cite[pp.118]{Camm}.
Since then magic squares have appeared in various cultures, and have been an object of curiosity
in art, philosophy, religion and mathematics. The study of magic squares with additional structure is a topic that has garnered great interest in both recreational and research mathematics.

The first $4 \times 4$ magic square of squares (Figure \ref{fig:euler}) was constructed by Euler, in a letter
sent to Lagrange in $1770$. Though Euler did not provide any explanation of how he constructed the square,
he presented his method to the St. Petersburg Academy of Sciences the same year; the construction is based on the observation that the product of two sums of four squares can itself be expressed as a sum of four squares.
This idea was used in $1754$ by Euler to make partial progress, which led Lagrange, in the same year as the letter, to the first complete proof of the \emph{four square theorem}: every positive integer is the sum of at most four squares \cite{boyer, VA}.

\begin{figure}[!h]\label{fig:euler}
\[
\begin{array}{|c|c|c|c|}
\hline
68^2 & 29^2 & 41^2 & 37^2\\
\hline
17^2 & 31^2 & 79^2 & 32^2 \\
\hline
59^2 & 28^2 & 23^2 & 61^2 \\
\hline
11^2 & 77^2 & 8^2  & 49^2\\
\hline
\end{array}\]
\caption{Euler's $4 \times 4$ magic square of squares with magic constant $8515$.}
\end{figure}

The search for a $3 \times 3$ magic square of squares was popularized by Martin Gardner in $1996$
offering a \$$100$ prize to the first person to construct such a square,
though the problem had already been posed by Edouard Lucas in $1876$ and Martin LaBar in $1984$ \cite{boyer, VA}.
Despite the great interest and efforts, the prize remains unclaimed (as do the \euro 100 and bottle of champagne offered by Boyer \cite{list} for the same problem). 
However, there are number of results making progress on this problem,
for which a comprehensive list can be found in \cite{list}.
A $3 \times 3$ magic square of squares gives rise to a rational point with nonzero coordinates on a surface
cut out by $6$ quadrics in the space $\mathbb{P}^{8}$. A deep conjecture of Lang predicts that this surface contains only finitely many curves of genus $0$ or $1$, and that outside of these curves it has only finitely many rational points.
The method of \cite{BTVA} establishes that indeed this surface contains only finitely many curves of genus $0$ or $1$.
In fact, we know the surface contains curves of genus $0$ or $1$ which do not correspond to magic squares (for instance, lines parametrising repeated entries). 
Therefore, it seems plausible that perhaps there are no $3 \times 3$ magic squares of squares, or that they are remarkably rare.

An $n \times n$ magic square of squares corresponds to a rational point on a variety cut out by $2n$ quadrics in the space $\mathbb{P}^{n^2 - 1}$. In contrast to the $n = 3$ case, it is quite reasonable from a geometric point of view that these spaces would carry many rational points for $n \geq 5$ \cite{VA}. This line of logic
has led V{\'a}rilly-Alvarado to make the following conjecture.

\begin{conjecture}[{\cite[Conjecture 4.3]{VA}}]\label{conj:tva}
There is a positive integer $n_0(2)$ such that for every integer $n\geq n_0(2)$, there exists an $n \times n$ magic square of squares.
\end{conjecture}

In the light of Euler's example above, V{\'a}rilly-Alvarado further suggested that the conjecture holds with $n_0(2) = 4$. We establish that this is indeed the case.

\begin{theorem}\label{main1}
For every integer $n\geq 4$, there exists an $n \times n$ magic square of squares.
\end{theorem}

In fact, we also establish a generalisation of the conjecture for higher powers.
\begin{theorem}\label{main1+}
Let $d \geq 3$.
There is a positive integer $n_0(d)$ such that for every integer $n\geq n_0(d)$, there exists an $n \times n$ magic square of $d^\textnormal{th}$ powers.
\end{theorem}

The novel feature of our work, in comparison to prior work on magic squares in the literature,
is that we are applying the Hardy-Littlewood circle method to a problem, %\footnote{We thank Trevor Wooley for pointing out to us that
%the application of the circle method to this problem has been part of the folklore for at least 30 years. [maybe not necessary?]},
for which all previous results had been based on constructive methods (see Remark \ref{remDF} regarding the independent work by Flores \cite{DF}). In particular, our result is non-constructive.
%constructing explicit examples.
We believe this is the first\footnote{Here we only mean in terms of magic squares with our definition. In a broader sense, as explained in Remark \ref{remDF},
it is second to Flores' paper which appeared on the arXiv on 12/6/2024 while our original version on 13/6/2024.} instance, in the literature, of the circle method being applied to study magic squares. Prior to our result, the existence of $n \times n$ magic squares of $d^\textnormal{th}$ powers was only known for sparse values of $n$ when $d \geq 4$
(for $n = q^d$ with $q \geq 2$ and additionally for small values of $n$ when $4 \leq d \leq 7$), and the existence was also known for $n$ in certain congruence classes when $d \in \{2, 3\}$, while we establish the result for all $n \geq n_0(d)$. We present a more detailed account of the progress on this topic in Section \ref{hist}.

As it can be seen from the proof, the statement of Theorem \ref{main1+} in fact holds with
\begin{eqnarray}
\label{n0002}
n_0(d) = \begin{cases}
  4 \min \{2^d, d (d+1)  \} + 20 & \mbox{if } 3 \leq d \leq 4, \\
  4 \lceil d (\log d + 4.20032) \rceil + 20 & \mbox{if } d \geq 5.
\end{cases}
\end{eqnarray}
It is worth mentioning that the smallest known example of magic squares of $d^\textnormal{th}$ powers was of size $2^d$ for $d \geq 8$;
therefore, we improve on the smallest size of magic squares of $d^\textnormal{th}$ powers known to exist from $2^d$ to $n_0(d)$ as in (\ref{n0002}) for $d \geq 9$.
%$$
%n_0(d) \leq 14 d (d+1) + 79 < 2^d.
%$$

The main challenge in applying the circle method to the equations defining magic squares is that they  define a
variety which is ``too singular'' for the method to be directly applicable; this is explained in Section \ref{appcircle}.
Though we can not apply the results for systems of general homogeneous forms by Birch \cite{birch} and Rydin Myerson \cite{simon},
we can apply the version available for systems of diagonal forms by Br\"{u}dern and Cook \cite{BC}. However, this result requires certain ``partitionability'' of the coefficient matrix, which amounts to finding a sufficient number of disjoint linearly independent subsets of the columns of the coefficient matrix; establishing this is the
main technical challenge of the paper. We prove an optimal (up to a constant) lower bound for this quantity in Theorem \ref{main2}.
%which is the main technical result of this paper.
%In Appendix \ref{app}, we turn the proof of Theorem \ref{main2} into an algorithm and establish a stronger lower bound by explicit computation for small values of $n$.
In this context, the circle method allow us to transform the problem of considerable complexity of finding magic squares of
$d^\textnormal{th}$ powers to a theoretically and computationally simpler problem of finding disjoint linearly independent subsets of the columns of the coefficient matrix.
We expect that this lower bound will be particularly useful for studying other variants of magic squares.
%In fact, in our follow-up work we apply the techniques developed in this paper to study multimagic squares (defined in Section \ref{hist}).
We believe that our Theorems \ref{main1} and \ref{main1+} are a significant step towards the complete classification of the existence of $n \times n$ magic squares of $d^\textnormal{th}$ powers.

\subsection{Progress on magic squares of powers}
\label{hist}
An $n \times n$ magic square is a $d$-multimagic square if it remains a magic square when all the entries are raised to the $i^\text{th}$ power for every $i = 2, \ldots, d$. It is called a normal $d$-multimagic square if it is a $d$-multimagic square with entries consisting of the numbers $1$ up to $d^2$. 
Clearly a normal $d$-multimagic square provides an example of a magic square of $d^\text{th}$ powers. The first published $2$-multimagic squares,
which were of sizes $8 \times 8$ and $9 \times 9$, were obtained by Pfeffermann in $1890$. The first $3$-multimagic square was obtained in $1905$ by Tarry.
In fact, Tarry was the first to devise a systematic method of constructing $2$-mulitmagic and $3$-multumagic squares \cite{Keed}.
Prior to our Theorem \ref{main1+}, all known examples of magic squares of $d^\text{th}$ powers for $d \geq 8$ in fact came from multimagic squares.
For $2$-multimagic squares it was proved by Chen, Li, Pan and Xu \cite{CLPX} that there exists a $2m \times 2m$  normal $2$-multimagic squares for all $m \geq 4$.
There is also a result by Hu, Meng, Pan, Su and Xiong \cite{HMPSX} which establishes the existence of $16 m \times 16 m$ normal $3$-multimagic squares for all $m \geq 1$.
We refer the reader to the introductions of \cite{CLPX}, \cite{HMPSX} and \cite{FH} for more detailed history on the progress regarding normal $2$ and $3$-multimagic squares.
The most general result regarding multimagic squares is given by Derksen, Eggermont and van der Essen \cite{DEV} who have proved that there exist
$n \times n$ normal $d$-multimagic squares for any $n = q^d$ with $q \geq 2$. There is also a comprehensive list of various magic squares recorded by Boyer \cite{list}
which includes the following:
\begin{itemize}
  \item  $n \times n$ magic squares of squares for $4 \leq n \leq 7$.
  \item  $n \times n$ $2$-multimagic squares for $8 \leq n \leq 64$.
  \item  number of magic squares of $d^\text{th}$ powers for $2 \leq d \leq 7$.
\end{itemize}
Thus we see that our Theorems \ref{main1} and \ref{main1+} are the first result of this type
and of different nature compared to all the previous results on magic squares in the literature.

\begin{remark}
\label{remDF}
It is important to mention the very nice work by Flores \cite{DF} here. 
It was completely unknown to us that he was simultaneously working on
an adjacent problem until his preprint appeared on the arXiv on 12/6/2024 which prompted us to post our original version on 13/6/2024. 
Flores established the existence of $n \times n$ ``non-trivial'' $d$-multimagic squares for all $n \geq 2 d(d+1)+1$
%(see (\ref{n0002}) for comparison)
also using the circle method. Strictly speaking, however, his non-trivial multimagic square is not a magic square in the sense of Conjecture \ref{conj:tva}, because it allows repeated entries; this is an important distinction for us, because for example $3 \times 3$ magic squares of squares with repeated entries are known to exist while the big open problem is regarding the existence of one without. Though it is not dealt with in his paper, there is no doubt that Flores' work can be adapted to deal with multimagic squares without repeated entries as well.
\end{remark}

\subsection{Application of the Hardy-Littlewood circle method}
\label{appcircle}
Let us label the entries of the $n \times n$ magic square of $d^\text{th}$ powers as follows:
\[
\begin{array}{|c|c|c|}
\hline
x_{1,1}^d & \ldots & x_{1,n}^d \\
\hline
\vdots & \vdots & \vdots \\
\hline
x_{n,1}^d & \ldots & x_{n,n}^d \\
\hline
\end{array}\,.\]
Let $\mu$ be a positive integer.
The system of equations defining the $n\times n$ magic square of $d^\textnormal{th}$ powers with magic constant $\mu$ is equivalent to
\begin{eqnarray}\label{eq:msos}
x_{1,1}^d + \cdots + x_{1,n}^d &=& \mu \\
                               &\vdots& \notag \\
x_{n,1}^d + \cdots + x_{n,n}^d &=& \mu\notag\\
x_{1,1}^d + \cdots + x_{n,1}^d &=& \mu\notag\\
                               &\vdots& \notag \\
x_{1,n-1}^d + \cdots + x_{n,n-1}^d &=& \mu\notag\\
x_{1,1}^d + \cdots + x_{n,n}^d &=& \mu \notag\\
x_{1,n}^d + \cdots + x_{n,1}^d &=& \mu, \notag  %\tag{$\boldsymbol{\ast}_d$}
\end{eqnarray}
which we denote by $\mathbf{F}_0(\x_0) = \boldsymbol{\mu}$.
A priori the system of equations defining the $n\times n$ magic square of $d^\textnormal{th}$ powers with magic constant $\mu$
requires $2n+2$ equations; however, it can be verified that there is one degree of redundancy
(the $2n-1$ equations corresponding to $n$ rows and $n-1$ columns imply the equation corresponding to the remaining column).
Therefore, this is a system of equations defined by  $R_0 = 2n+1$  degree $d$ homogeneous polynomials in $N_0 = n^2$ variables.
Since the number of variables $N_0$ grows quadratically with $R_0$, it may seem reasonable at first to expect that the known
results regarding the Hardy-Littlewood circle method readily apply to this system. As it turns out, the system (\ref{eq:msos}) is ``too singular'' for this to be the case.

A seminal result by Birch \cite{birch} establishes the existence of a non-trivial integral solution to a general system of equations
$F_1 (\x) = \cdots = F_R(\x) = 0$ defined by $R$ homogeneous polynomials of degree $d$ in $N$ variables. For $d = 2$, this was greatly improved by Rydin Myerson \cite{simon}.
An important quantity in these results is the dimension of the singular locus of the pencil
$$
\sigma_\RR (\mathbf{F}) = \max_{ \boldsymbol{\beta} \in \mathbb{R}^{R} } \dim  \{ \mathbf{x} \in \mathbb{A}^{N}:
\boldsymbol{\beta}.  \nabla \mathbf{F} (\x) = \mathbf{0} \},
$$
where $\mathbf{F} = (F_1, \ldots, F_R)$.
The required bound for $\sigma_\RR(\mathbf{F})$ is
$N > \sigma_\RR(\mathbf{F}) + (d-1) 2^{d-1} R (R+1)$ for Birch's result to be applicable, while $N > \sigma_\RR(\mathbf{F}) + 8 R$ for Rydin Myerson's result.
However, it can be verified that
$$
\sigma_\RR (\mathbf{F}_0) \geq n^2 - n = N_0 - \frac{R_0 - 1}{2},
$$
which is far too large to make use of either of the mentioned results.
Both of these results are for general systems of homogeneous polynomials; however, the system (\ref{eq:msos}) consists only of diagonal polynomials and
there are results in this direction as well. Let us now suppose that $F_1, \ldots, F_R$ are diagonal polynomials.
The system of diagonal equations $F_1 (\x) = \cdots = F_R(\x) = 0$ was first studied by Davenport and Lewis \cite[Lemma 32]{DavenportLewis};
their result required
$$
N \geq \begin{cases} \lfloor 9 R^2 d \log (3 R d) \rfloor & \mbox{if $d$ is odd}, \\
\lfloor 48 R^2 d^3 \log (3 R d^2) \rfloor & \mbox{if $d \geq 4$ is even},
\end{cases}
$$
to establish the existence of a non-trivial integral solution. Since $N_0 = (R_0 - 1)^2/4$, we can not hope to directly apply this result to (\ref{eq:msos}).
By incorporating the breakthrough on Waring's problem by Vaughan \cite{Va89},
Br\"{u}dern and Cook \cite{BC} improved the number of variables required to be
$$
N \geq 2 d R (\log  d + O(\log \log d) ).
$$
We remark that both of these results require a suitable ``rank condition'' on the coefficient matrix.
Since the time of these two papers, there has been great progress regarding Waring's problem
(for example, by Wooley \cite{Woooo92}, \cite{Woo95} and more recently by Wooley and Br\"{u}dern \cite{BruWoo}) and also the resolution of Vinogradov's mean value theorem
(see the work by Bourgain, Demeter and  Guth \cite{BDG}, and by Wooley \cite{WooCub, WooNest}).
By incorporating these recent developments, the required number of variables may be further improved.
Though this technical procedure is mostly standard, we present the details in our separate paper \cite{RomY}
and the result applied to the system (\ref{eq:msos}) is summarised in Theorem \ref{main2}.

\subsection*{Acknowledgements}
NR was supported by FWF project ESP 441-NBL while SY by a FWF grant (DOI 10.55776/P32428).
The authors are grateful to Tim Browning for his constant encouragement and enthusiasm,
%J\"{o}rg Br\"{u}dern and Trevor Wooley for very helpful discussions regarding his paper \cite{BC} and recent developments in Waring's problem, respectively,
J\"{o}rg Br\"{u}dern for very helpful discussion regarding his paper \cite{BC}
and Diyuan Wu for turning the proof of Theorem \ref{rank} in the original version into an algorithm and running the computation for us, for which the results are available in the appendix of the original version. They would also like to thank Christian Boyer for maintaining his website \cite{list} which contains a comprehensive list of various magic squares discovered, Brady Haran and Tony V{\'a}rilly-Alvarado for their public engagement activity of mathematics and magic squares of squares\footnote{A YouTube video ``Magic Squares of Squares (are PROBABLY impossible)'' of the Numberphile channel by Brady Haran, in which Tony V{\'a}rilly-Alvarado appears as a guest speaker: \url{https://www.youtube.com/watch?v=Kdsj84UdeYg}.}, and all the magic squares enthusiasts who have contributed to \cite{list} which made this paper possible. Finally, the authors would like to thank Daniel Flores for his work \cite{DF} which inspired them to optimise the proof of Theorem \ref{rank}
and
Trevor Wooley for very helpful discussion regarding recent developments in Waring's problem and his comments on the original version of this paper.
%for pointing out, though it was not aware to them, that the idea of using the circle method to study magic squares of powers had been around for at least 30 years.

\subsection*{Notation}
We make use of the standard abbreviations $e(z) = e^{2\pi i z}$ and $e_q(z) = e^{\frac{2 \pi i z}{q}}$.
Given a vector $\a = (a_1, \ldots, a_R) \in \ZZ^R$, by $1 \leq \a \leq q$ we mean $1 \leq a_i \leq q$ for each $1 \leq i \leq R$.

\section{Magic squares of powers}\label{sec:powers}
Let $\mu$ be a positive integer. Let $M_0$ be the coefficient matrix of the system (\ref{eq:msos}), which is given by (\ref{MMM0}).
We define
$$
\mathcal{I} = \{0\} \cup \{ (i_1,j_1;i_2,j_2):  (i_1,j_1) \neq (i_2, j_2), 1 \leq i_1, j_1, i_2, j_2 \leq n \}.
$$
For each $\sigma = (i_1, j_1; i_2, j_2) \in \mathcal{I} \setminus \{ 0 \}$, we denote $\x_\sigma$ to be $\x_0$ with $x_{i_2,j_2}$ removed
and $\mathbf{F}_\sigma (\x_\sigma)  = \boldsymbol{\mu}$ the system of equations obtained
by substituting $x_{i_2, j_2} = x_{i_1, j_1}$ into $\mathbf{F}_0(\x_0) = \boldsymbol{\mu}$.
We then denote by $M_{ \sigma }$ the coefficient matrix of this system, which is obtained from $M_0$
by adding the $((i_2 - 1)n + j_2)$-th column to the $((i_1 - 1)n + j_1)$-th column and then deleting the $((i_2 - 1)n + j_2)$-th column (here
$((i - 1)n + j)$-th column corresponds to the $x_{i,j}$ variable).
%by replacing the $x_{i_1,j_1}$ column with the sum of the $x_{i_1,j_1}$ and $x_{i_2,j_2}$ columns, then removing the $x_{i_2,j_2}$ column.

For $\sigma \in \mathcal{I}$, $\mathfrak{B} \subseteq \NN$ and  $X \geq 1$, we introduce the following counting function
\begin{align*}
N_\sigma (\mathfrak{B}; X, \mu) = \#\{ \x_\sigma \in (\mathfrak{B} \cap [1, X])^{n^2 - \epsilon (\sigma) }:  \mathbf{F}_\sigma (\x_\sigma) = \boldsymbol{\mu}  \},
\end{align*}
where
$$
\epsilon (\sigma) = \begin{cases}
                      0 & \mbox{if } \sigma = 0, \\
                      1 & \mbox{if } \sigma \in \mathcal{I} \setminus \{ 0 \}.
                    \end{cases}
$$
Then the number of magic squares, whose entries are restricted to $\mathfrak{B}$, with magic constant $\mu$ is given by
\begin{eqnarray}
\label{eqn}
N(\mathfrak{B}; X, \mu) =  N_0(\mathfrak{B}; X, \mu) + O \left(  \sum_{\sigma \in  \mathcal{I} \setminus \{ 0 \} } N_\sigma (\mathfrak{B}; X, \mu) \right).
\end{eqnarray}
In this expression, the sum in the $O$-term is the contribution from squares whose entries are not distinct.
We estimate these counting functions using the Hardy-Littlewood circle method for $\mathfrak{B} = \NN$ and
$$
\mathcal{A}(X, X^{\eta}) = \{ x \in [1, X] \cap \NN: \textnormal{prime } p| x \textnormal{ implies } p \leq X^{\eta}  \}
$$
for sufficiently small $\eta > 0$. The optimal lower bound for $n_0(d)$ in Theorem \ref{main1+} will come from using smooth numbers for most choices of $d$, while  
for $d \in \{2, 3, 4 \}$ it will come from  the natural numbers instead. 

\begin{definition}
\label{TTT}
Let $\sigma \in \mathcal{I}$. We define $\Psi ({M_\sigma})$ to be the largest integer $T$ such that there exists
$$
\{ \mathfrak{D}_1, \ldots, \mathfrak{D}_T \},
$$
where each $\mathfrak{D}_i$ is a linearly independent set of $2n+1$ columns of $M_\sigma$ and $\mathfrak{D}_i \cap \mathfrak{D}_j \neq \emptyset$ if $i \neq j$.
We let $T_\sigma$ be a non-negative integer such that
$$
\Psi ({M_\sigma}) \geq T_{\sigma}.
$$
\end{definition}

By applying the circle method to the system of diagonal equations (\ref{eq:msos}),
we obtain the following as a direct consequence of \cite[Theorems 1.4 and 1.5]{RomY}.
Here, we present a simplified version of the result (see \cite[Tables 1 and 2]{RomY} for more accurate values and also \cite[Remark 2.3]{RomY}).
\begin{theorem}
\label{main2}
Let $\mathfrak{B} = \NN$ or $\mathcal{A}(X, X^{\eta})$ with $\eta > 0$ sufficiently small.
Suppose
$$
\min_{\sigma \in \mathcal{I}} T_\sigma \geq
\begin{cases}
  \min \{ 2^d, d (d+1) \} +  1  & \mbox{if } 2 \leq d \leq 4  \mbox{ and }\mathfrak{B} = \NN,\\
  \lceil d (\log d + 4.20032) \rceil +  1 & \mbox{if } d \geq 5 \mbox{ and } \mathfrak{B} = \mathcal{A}(X, X^{\eta}) .
\end{cases}
$$
Then there exists $\lambda > 0$ such that
$$
N_0(\mathfrak{B}; X, \mu) = C_\mathfrak{B} \mathfrak{S} \mathfrak{I} X^{n^2 - (2n+1)d} + O(  X^{n^2  - (2n+1)d}  (\log X)^{- \lambda} ),
$$
where
$$
C_\mathfrak{B}
= \begin{cases}
    1 & \mbox{if } \mathfrak{B} = \NN, \\
    c(\eta) & \mbox{if } \mathfrak{B} = \mathcal{A}(X, X^{\eta}),
  \end{cases}
$$
$c(\eta) > 0$ is a constant depending only on $\eta$, the singular series $\mathfrak{S}$ and the singular integral $\mathfrak{I}$ are defined in (\ref{defSS}) and (\ref{defSI}), respectively.
Furthermore,
$$
\max_{\sigma \in \mathcal{I} \setminus \{0\} }  N_\sigma (\mathfrak{B}; X, \mu) \ll  X^{n^2 - (2n+1)d - 1}.
$$
\end{theorem}

The next step is to establish lower bounds for $\mathfrak{S}$ and $\mathfrak{I}$ which are independent of $X$.
We present the details of the following proposition in Section \ref{prop:SS}.
\begin{proposition}
\label{SSSS}
Given any $X \geq 1$ sufficiently large, there exists $\mu = \mu(X) \in \NN$ such that
$$
\mathfrak{S} \mathfrak{I} > \mathfrak{c}_{n, d},
$$
where $\mathfrak{c}_{n, d} > 0$ is a constant depending only on $n$ and $d$.
\end{proposition}

Finally, in order to establish the existence of magic squares of $d^\textnormal{th}$ powers with magic constant $\mu$, it remains to establish a lower bound for $T_\sigma$.
The following is the main technical result of the paper, which we prove through Sections \ref{sec:linear}--\ref{Even}.
\begin{theorem}
\label{rank}
Let $n \geq 8$. Then
$$
T_0 \geq \left \lfloor \frac{n}{4} \right \rfloor - 1.
$$
\end{theorem}

By combining these results collected in this section, we obtain Theorem \ref{main1+} as an immediate consequence.
\begin{proof}[Proof of Theorem \ref{main1+}]
It follows from Theorem \ref{rank} that
$$
\min_{\sigma \in \mathcal{I} \setminus \{ 0 \} } T_\sigma \geq T_0 - 2 \geq \left \lfloor \frac{n}{4} \right \rfloor - 3.
$$
Therefore, by combining this estimate, Theorem \ref{main2} and Proposition \ref{SSSS} with (\ref{eqn}), we obtain that
Theorem \ref{main1+} holds with
\begin{eqnarray}
\label{n000}
n_0(d) = \begin{cases}
  4 \min \{2^d, d (d+1)  \} + 20 & \mbox{if $2 \leq d \leq 4$,} \\
  4 \lceil d (\log d + 4.20032) \rceil + 20 & \mbox{if $d \geq 5$.}
\end{cases}
\end{eqnarray}
\end{proof}

In order to establish Theorem \ref{main1} we need to deal with smaller values of $n$.
%In Appendix \ref{app}, we will turn the proof of Theorem \ref{rank} into an algorithm and establish a lower bound for $\min_{\sigma \in \mathcal{I}} T_\sigma$ computationally for these $n$.
\begin{proof}[Proof of Theorem \ref{main1}]
It follows from the proof of Theorem \ref{main1+} above that
there exists an $n \times n$ magic square of squares as soon as $n \geq 36$.
Since the statement is already known for $4 \leq n \leq 64$, in fact explicit examples have been discovered and listed in \cite{list}, this completes the proof.
\end{proof}

\section{Singular series and singular integral: Proof of Proposition \ref{SSSS}}
\label{prop:SS}
We let $\textnormal{Col}(M_0)$ denote the set of columns of $M_0$.
Let $\textnormal{Jac}_{\mathbf{F}_0}$ denote the Jacobian matrix of $\mathbf{F}_0$
and $\mathbb{F} = \RR$ or $\QQ_p$ for any prime $p$.
A crucial fact we make use of in this section is that given any $z \in \mathbb{F} \setminus \{0\}$, we have
$$
\textnormal{Jac}_{\mathbf{F}_0} (z, \ldots, z) = d z^{d-1} M_0,
$$
which in particular is of full rank. Therefore, if $\mathbf{F}_0 ( z, \ldots, z ) = \boldsymbol{\mu}$,
then it is in fact a non-singular solution.

Let
$$
S(q,a) =  \sum_{1 \leq x \leq q} e_q (a x^d).
$$
We define the singular series
\begin{eqnarray}
\label{defSS}
\mathfrak{S}  = \sum_{q =1}^\infty A (q),
%q^{-n^2}  \sumstar_{\boldsymbol{a} \bmod q}  A(q, \mathbf{a}) \,  e_q \left(  -  M \sum_{k = 1}^{2n + 1} a_k \right),
\end{eqnarray}
where
$$
A (q) = q^{- n^2  }  \sum_{ \substack{ 1 \leq  \a \leq  q \\ \gcd(q, \a) = 1  } } \prod_{\c \in \textnormal{Col}( M_0 ) } S(q, \a.\c) \cdot  e_q \left(  -  \mu \sum_{i = 1}^{2n + 1} a_i \right).
%\sum_{\x \bmod q} e_q ( (\mathbf{a} \Lambda).  \mathbf{x}^d - M \a).
$$
We have the following lemma regarding $\mathfrak{S}$  which is \cite[Lemmas 4.1 and 4.2]{RomY}.
\begin{lemma}
Suppose
$$
T_0 > \frac{d (2n+2)}{2n + 1}.
$$
Then
\begin{eqnarray}
\label{3.1}
A(q)  \ll  q^{-  (2n+1)( \frac{T_0}{d} - 1  )}, 
\end{eqnarray}
where the implicit constant is independent of $\mu$,  
and the series (\ref{SSSS}) converges absolutely. In fact,
$$
\mathfrak{S} = \prod_{p \textnormal{ prime} } \chi(p),
$$
where
$$
\chi(p) = 1 + \sum_{k = 1}^{\infty} A (p^k).
$$
\end{lemma}

We establish the desired lower bound for $\mathfrak{S}$ for special values of $\mu$.
\begin{lemma}
\label{serbdd}
Let $\mu = n p_0^d$, where $p_0$ is a prime number sufficiently large with respect to $n$ and $d$. Then
$\mathfrak{S}_0  > c_{n, d}$, where $c_{n, d} > 0 $ is a constant depending only on $n$ and $d$.
\end{lemma}

\begin{proof}
An application of (\ref{3.1}) yields
$$
|\chi(p) - 1| \ll \sum_{k = 1}^{\infty} p^{- k (2n + 1) (\frac{T_0}{d} - 1)} \ll p^{-  (2n + 1) (\frac{T_0}{d} - 1)}.
$$
Therefore, there exists $P > 0$ such that
\[
\prod_{p > P} \chi(p) > \frac{1}{2}.
\]
We remark that $P$ is independent of $\mu$.
The standard argument shows that
\[
\chi(p) = \lim_{m \rightarrow \infty} \frac{\nu_\mu(p^m)}{p^{m(n^2 -2n-1)}},
\]
where $\nu_\mu(p^m)$ is the number of solutions to the congruence $\mathbf{F}_0 (\x_0) \equiv \boldsymbol{\mu} \bmod{ p^m }$.
Since $\mu = n p_0^d$ with prime $p_0$ sufficiently large and $p_0$ is invertible modulo $p$ for $p \leq P$, it follows that
$$
\nu_\mu(p^m) = \nu_n(p^m).
$$
We know that $\x_0 = (1, \ldots, 1) \in \ZZ_p^{n^2}$ is a non-singular solution to
the system of equations $\mathbf{F}_0 (\x_0) = \boldsymbol{n}$.
Thus it follows from Hensel's lemma that
$$
\chi(p) = \lim_{m \rightarrow \infty} \frac{  \nu_\mu(p^m)  }{p^{m(n^2-2n-1)}}
=
\lim_{m \rightarrow \infty} \frac{  \lambda_n(p^m)  }{p^{m(n^2-2n-1)}}  > 0
$$
for $p \leq P$.
Therefore, we obtain
$$
\mathfrak{S}  >  \frac12  \prod_{p \leq P}   \chi(p)   \gg 1,
$$
where the implicit constant is independent of $\mu$.
 \end{proof}

Let
$$
I(\beta)
= \int_0^1 e(\beta \xi^d) \mathrm{d}\xi.
$$
We define the singular integral
\begin{eqnarray}
\label{defSI}
\mathfrak{I} =  \int_{\RR^{2n + 1}}   \prod_{ \c \in \textnormal{Col}(M_0) } I(\bgamma.\c)  \cdot  e \left(   -  \frac{\mu}{X^d} \sum_{i = 1}^{2n + 1} \gamma_i \right)  \mathrm{d} \bgamma.
\end{eqnarray}

We have the following lemma regarding $\mathfrak{I}$  which is \cite[Lemmas 4.3]{RomY}
\begin{lemma}\label{lem:singint}
Suppose $T_0 > d$. Then the integral (\ref{defSI}) converges absolutely.
\end{lemma}

We establish the desired lower bound for $\mathfrak{I}$ for special values of $\mu$.
\begin{lemma}
\label{intbdd}
Let $\varepsilon_0 > 0$ be sufficiently small. Suppose $\mu = n \zeta^d X^d$ with $\zeta \in [\varepsilon_0, 1 - \varepsilon_0]$.
Then $\mathfrak{I} > c_{n, d, \varepsilon_0}$, where $c_{n, d, \varepsilon_0} > 0$ is a constant depending only on $n, d$ and $\varepsilon_0$.
\end{lemma}
\begin{proof}
It follows by the standard argument as in \cite{MR0693325} that
$$
\mathfrak{I}  =
\lim_{L \rightarrow \infty}
\int_{[0,1]^{n^2}} \prod_{i = 1}^{2n + 1} \Phi_L(  F_{0,i}(\x_0) - \mu X^{-d})  \mathrm{d} \x_0,
$$
where
$$
\Phi_L(\eta) =
\begin{cases}
  L(1 - L|\eta|) & \mbox{if } |\eta| \leq L^{-1}, \\
  0 & \mbox{otherwise. }
\end{cases}
$$
Furthermore, since $\mu X^{-d } = n \zeta^d$ with $\zeta \in [\varepsilon_0, 1 - \varepsilon_0]$, we see that
$\x_0 = (\zeta, \ldots, \zeta)$ is a non-singular solution to the system of equations $\mathbf{F}_0(\x_0) = \boldsymbol{\mu}$.
It then follows by an application of the implicit function theorem as in \cite[Lemma 2]{MR0693325} that
$\mathfrak{I} > 0$. In particular, since $\zeta \in [\varepsilon_0, 1 - \varepsilon_0]$, the lower bound is independent of $\mu$.
\end{proof}

Let $\varepsilon_0 > 0$ be sufficiently small. By the prime number theorem, for any $X$ sufficiently large there exists a prime $p_0$ satisfying
$$
\varepsilon_0 X <   p_0 < (1 - \varepsilon_0) X.
$$
Then for $\mu = n p_0^d$ we see that Lemmas \ref{serbdd} and \ref{intbdd} hold, from which Proposition \ref{SSSS} follows.

\section{Linear equations}
\label{sec:linear}
In this section, we record a simple result regarding certain systems of linear equations 
which we will need in the following Section \ref{Bound for T}. 
Let $1 \leq i_1 < m < m + 1 < i_2 \leq n$. We denote by $\mathbb{L}_n(i_1; m; i_2)$ the system of $2n$ linear equations
\begin{eqnarray}
x_i  +  y_i &=& 0 \quad (1 \leq  i \leq  n)
\notag
\\
\notag
x_{i + 1 } + y_{i}  &=& 0 \quad  (i \neq  i_1 - 1,  i_1,  i_2 - 1,  i_2, m)
\\
\notag
x_{i_2  } + y_{ i_1  - 1 } &=& 0
\\
\notag
x_{i_2 + 1} + y_{ i_1 }   &=& 0
\\
\notag
x_{ i_1   } + y_{ i_2 - 1 }  &=& 0
\\
\notag
x_{ i_1 + 1  } + y_{i_2 } &=& 0
\\
\notag
y_{m} &=& 0,
\end{eqnarray}
where $x_{n+1}$ is identified as $x_1$.

\begin{lemma}
\label{First}
Suppose $1 < i_1 <  m < m + 1 < i_2 < n$. Then
\begin{eqnarray}
\notag
\{ (\x, \y) \in \RR^{2n}: (\x,\y) \textnormal{ satisfies }   \mathbb{L}_n(i_1; m; i_2) \} = \{ \mathbf{0}  \}.
\end{eqnarray}
\end{lemma}
\begin{proof}
Consider the system of equations $\mathbb{L}_n(i_1; m; i_2)$.
By substituting $y_m = 0$ into the equation $x_m + y_m = 0$ and following through with the consequences,
we obtain
$$
x_{i} = y_i = 0 \quad (i \in [i_1 + 1, m]).
%x_{i_1 + 1} = \cdots = x_{m} = 0
%\quad
%\textnormal{and}
%\quad
%y_{i_1 + 1} = \cdots = y_{m} = 0.
$$
We can then deduce that
$$
0 = x_{i_1 + 1} = y_{i_2} = x_{i_2} = y_{i_1 - 1} = x_{i_1 - 1},
$$
which further implies that
$$
x_{i} = y_i = 0 \quad (i \in [1, i_1 - 1]).
%x_{1} = \cdots = x_{i_1 - 1} = 0
%\quad
%\textnormal{and}
%\quad
%y_{1} = \cdots = y_{i_1 - 1} = 0.
$$
Therefore, the system becomes
\begin{eqnarray}
\notag
x_{i_1} + y_{i_1} = 0 && x_{i_2 + 1} + y_{i_1 } = 0
\\
\notag
x_{m + 1} + y_{m + 1} = 0 && x_{m + 2} + y_{m + 1} = 0
\\
\notag
&\vdots&
\\
\notag
x_{i_2 - 2} + y_{i_2 - 2} = 0 && x_{i_2 - 1} + y_{i_2 - 2} = 0
\\
\notag
x_{i_2 - 1} + y_{i_2 - 1} = 0 && x_{i_1} + y_{i_2 - 1} = 0
\\
\notag
x_{i} + y_{i} = 0 && x_{i + 1} + y_{i} = 0 \quad (i \in [i_2 + 1, n]).
\end{eqnarray}
By substituting $x_1 = 0$, we obtain
$$
x_{i} = y_i = 0 \quad (i \in [i_2 + 1, n]).
%x_{i_2 + 1} = \cdots = x_{n} = 0
%\quad
%\textnormal{and}
%\quad
%y_{i_2 + 1} = \cdots = y_{n} = 0,
$$
which further implies that
$$
0 = x_{i_2 + 1} = y_{i_1} = x_{i_1} = y_{i_2 - 1}.
$$
Finally, by substituting $y_{i_2 - 1} = 0$ into the remaining system, we see that
$$
x_{i} = y_i = 0 \quad (i \in [m + 1, i_2 - 1])
%x_{m + 1} = \cdots = x_{i_2 - 1} = 0
%\quad
%\textnormal{and}
%\quad
%y_{m + 1} = \cdots = y_{i_2 - 1} = 0
$$
as desired.
\end{proof}

\section{Preliminaries}
%\section{Preparation for the proof of Theorem \ref{rank}}
\label{sec:prep}

Let $I_n$ denote the $n \times n$ identity matrix and define $\widetilde{I}_n$ to be the $(n-1) \times n$ matrix obtained by removing the last row from $I_n$.
Then the $(2n+1) \times n^2$  coefficient matrix of the system (\ref{eq:msos}) is
\begin{eqnarray}
\label{MMM0}
M_0 =
\begin{bmatrix}
  \begin{bmatrix}
  1 & 1 & \cdots & 1 \\
  0 & 0 & \cdots & 0 \\
  \vdots & \vdots & \vdots & \vdots \\
  0 & 0 & \cdots & 0
  \end{bmatrix}
     & \begin{bmatrix}
  0 & 0 & \cdots & 0 \\
  1 & 1 & \cdots & 1 \\
  \vdots & \vdots & \vdots & \vdots \\
  0 & 0 & \cdots & 0
  \end{bmatrix} & \cdots  & \begin{bmatrix}
  0 & 0 & \cdots & 0 \\
  0 & 0 & \cdots & 0 \\
  \vdots & \vdots & \vdots & \vdots \\
  1 & 1 & \cdots & 1
  \end{bmatrix} \\
  \widetilde{I}_n & \widetilde{I}_n & \cdots  & \widetilde{I}_n \\
  \begin{bmatrix}
  1 & 0 & \cdots & 0 \\
  0 & 0 & \cdots & 1
  \end{bmatrix} &
  \begin{bmatrix}
  0 & 1 & \cdots & 0 & 0 \\
  0 & 0 & \cdots &  1 & 0
  \end{bmatrix}
   & \cdots & \begin{bmatrix}
  0 & 0 & \cdots & 1 \\
  1 & 0 & \cdots & 0
  \end{bmatrix}
\end{bmatrix}.
\end{eqnarray}
We shall denote by $\c_{i,j}$ the $((i-1)n + j)$-th column of $M_0$, for each $1 \leq i, j \leq n$.
The first subscript $i$ will always satisfy $1 \leq i \leq n$, but for the second subscript $j$, to simplify our exposition, we will consider it modulo $n$, that
is
$$
c_{i, j + m n} = c_{i, j},
$$
for any  $1 \leq j \leq n$ and $m \in \mathbb{Z}$.
\begin{definition}
We shall refer to the set  $\{ \c_{i,1}, \ldots, \c_{i,n} \}$ as the $i$-th block for each $1 \leq i \leq n$.
\end{definition}
Given $\mathfrak{B} \subseteq \textnormal{Col}(M_0)$, we define
\begin{eqnarray}
\label{B[i]}
\mathfrak{B}[i] = \{ 1 \leq j \leq n: \c_{i, j} \in \mathfrak{B}  \}
\end{eqnarray}
for each $1 \leq i \leq n$.

Let
$$
\mathcal{S} = \mathcal{S}_1 \cup \mathcal{S}_2,
$$
where
$$
\mathcal{S}_1 = \{ \c_{1,1}, \c_{2,2}, \ldots , \c_{n - 1, n - 1} , \c_{n,n} \}
$$
and
$$
\mathcal{S}_2 = \{ \c_{1,n}, \c_{2,n-1}, \ldots  , \c_{n-1, 2}  , \c_{n, 1}  \}  \setminus \mathcal{S}_1.
$$
The columns in $\mathcal{S}_1$ are precisely the ones with $1$ in the $2n$-th entry and
$\mathcal{S}_2$ the ones with $1$ in the $(2n+1)$-th entry that are not contained in $\mathcal{S}_1$.
We remark that
$$
\# \mathcal{S}_2
= \begin{cases}
    n  & \mbox{if $n$ is even},   \\
    n - 1 & \mbox{if $n$ is odd}.
  \end{cases}
$$

Let us define
\begin{eqnarray}
\label{defep}
\epsilon(n) = \begin{cases}
                1 & \mbox{if $n$ is even}, \\
                0 & \mbox{if $n$ is odd},
              \end{cases}
\end{eqnarray} %$1 \leq   N_1 < N_2  \leq n$ to be chosen in due course. In particular, we choose it so that
and
\begin{eqnarray}
\label{cond1}
N  = \frac{n  - 1  + \epsilon(n)  }{2 } - 2.
%\begin{cases}
%  \frac{n}{2} - 2 & \mbox{if $n$ is even},  \\
%  \frac{n - 1}{2} - 2 & \mbox{if $n$ is odd}.
%\end{cases}
\end{eqnarray}
For each $0 \leq \ell \leq N$,  we define
$$
\widetilde{ \mathfrak{B}  }_\ell
= \{
\c_{1, 1 + 2 \ell }, \c_{2, 2 +  2  \ell  },  \ldots,  \c_{n, n + 2 \ell }  \} \bigcup  \left(  \{  \c_{1, 1 + (2 \ell + 1)  }, \c_{2, 2 +  (2 \ell + 1) }, \ldots,  \c_{n, n +   (2 \ell + 1)  }
\} \setminus \{ \c_{n - 2 \ell - 1,  n}  \} \right).
$$
This set contains precisely one column of the form $\c_{*, n}$, which is
$\c_{n- 2 \ell, n }$, a column whose entries between
$(n+1)$-th  and $(2n-1)$-th position are all $0$. 
We will use these sets $\widetilde{ \mathfrak{B}  }_\ell$ to construct the collection of pairwise disjoint sets of $2n + 1$ linearly independent columns of $M$ necessary to complete the proof of Theorem \ref{rank}.

By (\ref{cond1}) it follows that
\begin{eqnarray}
\label{pdis}
\widetilde{ \mathfrak{B}  }_\ell \cap \widetilde{ \mathfrak{B}  }_{\ell'} = \emptyset \quad (0 \leq \ell < \ell' \leq N).
\end{eqnarray}
 %It is also clear that here
%$$
%\c_{*, n} = \c_{n - 2 \ell, n}.
%$$

\begin{lemma}
\label{badvec}
Given any $1 \leq \ell \leq N$, we have
$$
\widetilde{ \mathfrak{B}  }_\ell \cap \mathcal{S}  = \left\{ \c_{\frac{n +  1 - \epsilon(n)}{2}  - \ell,  \frac{n + 1 - \epsilon(n)}{2} + \epsilon(n) + \ell }, \c_{n  - \ell, \ell + 1}   \right\}.
$$
\end{lemma}
\begin{proof}
Since $\mathcal{S}_1 \subseteq \widetilde{ \mathfrak{B}  }_0$, it follows from (\ref{pdis}) that
$$
\widetilde{ \mathfrak{B}  }_\ell \cap \mathcal{S}_1 = \emptyset.
$$
Therefore, if $\widetilde{ \mathfrak{B}  }_\ell \cap \mathcal{S} \neq \emptyset$, then there  exist $\epsilon \in \{0,1\}$ and $ 1 \leq i, j \leq n$ such that
$$
(i, i + 2 \ell  + \epsilon + m(i, \ell, \epsilon) n) = (j, n + 1 - j),
$$
where
$$
m(i, \ell, \epsilon)
= \begin{cases}
      0 & \mbox{if } 1 \leq  i + 2 \ell  + \epsilon \leq n, \\
      -1 & \mbox{if }  n < i + 2 \ell  + \epsilon \leq 2n.
    \end{cases}
$$
It follows that $i = j$ and
$$
j + 2 \ell  + \epsilon +  m(j, \ell, \epsilon) n = n + 1 - j,
$$
or equivalently
$$
j =  \frac{n + 1 - \epsilon}{2}  - \ell - \frac{m(j, \ell, \epsilon)  n}{2}.
$$

Suppose  $m(j, \ell, \epsilon) = 0$. Then it must be that $\epsilon = 1$ if $n$ is even and $0$ if $n$ is odd.
Thus we see that
$$
\c_{\frac{n + 1 - \epsilon(n) }{2}   - \ell,  \frac{n + 1 - \epsilon (n)}{2} + \epsilon(n) + \ell } \in \widetilde{ \mathfrak{B}  }_\ell \cap \mathcal{S};
$$
here we note that $1 \leq \frac{n + 1 - \epsilon(n)}{2} + \epsilon(n) + \ell  \leq n$.

On the other hand, suppose $m(j, \ell, \epsilon) = -1$. Then
$$
j =  n + \frac{1 -  \epsilon}{2}   - \ell.
$$
In particular, it must be that $\epsilon = 1$. Thus we see that
$$
\c_{n   - \ell, \ell + 1} \in \widetilde{ \mathfrak{B}  }_\ell \cap \mathcal{S}.
$$
Finally, we have $\c_{\frac{n + 1 - \epsilon (n) }{2}   - \ell,  \frac{n + 1 - \epsilon (n) }{2} + \epsilon (n) + \ell } \neq \c_{n   - \ell, \ell + 1 }$,
since
$$
1 \leq  \frac{n + 1 - \epsilon(n) }{2}   - \ell  \leq      \frac{n}{2} - \ell   <  n  - \ell.
$$
%just need n > 1
\end{proof}
As a consequence of this lemma, there are precisely two columns in $\widetilde{ \mathfrak{B}  }_\ell$
whose last two entries are not $0$.  In the next section, we replace these two columns from $\widetilde{ \mathfrak{B}  }_\ell$ with appropriate columns
whose last two entries are $0$ such that the resulting set is linearly independent. Finally, we complete the sets 
by adding two columns whose last two entries are not $0$ which preserve the linear independence.

%In the final section, we will complete the sets with columns containing last two entries that are non-zero.  

%We shall consider the two cases when $n$ is even and odd separately in Sections \ref{Even} and \ref{Odd} respectively.

\section{Proof of Theorem \ref{rank}}
\label{Even}
Let $n \geq 8$ and $1 \leq \ell \leq N$, where $N$ is defined in (\ref{cond1}).
Our goal is to replace the two columns in $\widetilde{ \mathfrak{B}  }_\ell$ from Lemma \ref{badvec}
in a suitable manner so that all columns have $0$ for the last two entries.
Let us denote
$$
i_1(\ell) = \frac{n + 1 - \epsilon(n) }{2}   - \ell
\quad
\textnormal{and}
\quad
i_2(\ell) =  n   - \ell,
$$
where $\epsilon(n)$ is defined in (\ref{defep}).
Then the two columns we need to replace are
\begin{eqnarray}
\notag %\label{even}
\c_{\frac{n + 1 - \epsilon(n) }{2}  - \ell,  \frac{n + 1 - \epsilon(n) }{2} + \epsilon(n)  + \ell }
=
\c_{i_1(\ell),  i_1(\ell) +  2 \ell + \epsilon(n) }
\quad
\textnormal{and}
\quad
\c_{n  - \ell, \ell + 1} = \c_{i_2(\ell), i_2(\ell) + 2 \ell + 1  }.
\end{eqnarray}
We shall modify $\widetilde{ \mathfrak{B}  }_\ell$ as follows
\begin{eqnarray}
\notag %\label{def:B}
\mathfrak{B}_\ell
&=&  \widetilde{ \mathfrak{B}  }_\ell \setminus
\left(
\{ \c_{ i_1(\ell),   i_1(\ell) +  2 \ell   }, \c_{  i_1(\ell),   i_1(\ell) +  2 \ell + 1}  \}
\bigcup
\{  \c_{ i_2(\ell),  i_2(\ell) + 2 \ell  }, \c_{ i_2(\ell),  i_2(\ell) + 2 \ell + 1 }   \} \right)
\\
\notag
&&\bigcup
\{ \c_{ i_1(\ell),   i_2(\ell) + 2 \ell  }, \c_{ i_1(\ell),   i_2(\ell) + 2 \ell + 1}  \}
\bigcup
\{  \c_{ i_2(\ell),  i_1(\ell) +  2 \ell   }, \c_{  i_2(\ell) ,  i_1(\ell) +  2 \ell + 1}   \},
\end{eqnarray}
that is we first remove the two columns in $\widetilde{\mathfrak{B}}_\ell$ from the $i_1(\ell)$-th and $i_2(\ell)$-th block
and then add back in two columns from each of these blocks with switched positions. %within the respective blocks switched.
We remark that we know there are two columns in  $\widetilde{\mathfrak{B}}_\ell$ from the $i_1(\ell)$-th and $i_2(\ell)$-th block, because
\begin{eqnarray}
\label{relnn}
1 < i_1(\ell) < n - 2 \ell - 1 < n - 2 \ell < i_2(\ell) < n.
\end{eqnarray}
%need \ell < n/2 - 1
It is clear that given $1 \leq \ell < \ell' \leq N$ we have
\begin{eqnarray}
\label{modiff}
\{ i_1(\ell) , i_2(\ell) \} \cap  \{ i_1(\ell') , i_2(\ell') \} = \emptyset.
\end{eqnarray}
Therefore, the two blocks of $\widetilde{\mathfrak{B}}_\ell$ which get modified to construct $\mathfrak{B}_\ell$ are unique to $\ell$.
Let us also recall the notation (\ref{B[i]}) and record the relations
\begin{eqnarray}
\label{Bind}
\mathfrak{B}_\ell [ i_1(\ell) ] = \{ i_2(\ell) +  2 \ell, i_2(\ell) +  2 \ell + 1  \}
\quad
\textnormal{and}
\quad
\mathfrak{B}_\ell [ i_2(\ell) ] = \{ i_1(\ell) +  2 \ell, i_1(\ell) +  2 \ell + 1  \}
\end{eqnarray}
for each $1 \leq \ell \leq N$.

\begin{lemma}
Given $1 \leq \ell \leq N$ such that
$$
\ell \neq
\begin{cases}
  \lfloor n / 4  \rfloor & \mbox{if $n$ is even},  \\
  \lfloor n / 4  \rfloor, \lfloor n / 4 \rfloor + 1 & \mbox{if $n$ is odd},
\end{cases}
$$
we have
$$
\{ \c_{i_1(\ell), i_2(\ell) + 2 \ell  }, \c_{i_1(\ell),  i_2(\ell) + 2 \ell + 1}  \}
\bigcap
\mathcal{S} = \emptyset
$$
and
$$
\{  \c_{i_2(\ell), i_1(\ell) +  2 \ell   }, \c_{i_2(\ell), i_1(\ell) +  2 \ell + 1 }   \}
\bigcap
\mathcal{S} = \emptyset.
$$
In particular, every column of $\mathfrak{B}_\ell$ has $0$ for the last two entries.
\end{lemma}
\begin{proof}
Let $\epsilon \in \{0,1\}$.
Since it can be verified that
\begin{eqnarray}
\notag
i_1(\ell) + i_2(\ell) + 2 \ell + \epsilon
\equiv  \frac{n + 1 - \epsilon(n)}{2}  + \epsilon
\not \equiv n + 1    \pmod{n}
\end{eqnarray}
and
\begin{eqnarray}
\notag
\pm ( i_1(\ell) - i_2(\ell) )
\equiv \pm  \frac{n + 1 - \epsilon(n)}{2}
\notag
\not \equiv  2 \ell + \epsilon     \pmod{n},
\end{eqnarray}
the result follows.
\end{proof}

Let
$\mathcal{Z} \subseteq \{1, \ldots,  N \}$ be a set with the following property:
given any $m, \ell \in \mathcal{Z}$, we have
\begin{eqnarray}
\label{propZ}
\frac{n + 1 - \epsilon(n) }{2} \not \equiv 2(m - \ell) + \delta \pmod{n}
\end{eqnarray}
for any
$\delta \in \{-1, 0, 1\}$.

Let us define
$$
\mathfrak{R} = \bigcup_{\ell \in \mathcal{Z}}  \widetilde{\mathfrak{B}}_{\ell}.  %\widetilde{\mathfrak{B}}_{N_1} \cup   \cdots \cup  \widetilde{\mathfrak{B}}_{N_2}.
$$

\begin{lemma}
\label{disj}
Given $\ell \in \mathcal{Z}$, we have
$$
\mathfrak{B}_\ell [ i_1(\ell) ]
\cap
\mathfrak{R}  [ i_1(\ell) ] = \emptyset
$$
and
$$
\mathfrak{B}_\ell [ i_2(\ell) ]
\cap
\mathfrak{R}  [ i_2(\ell) ]  =  \emptyset.
$$
\end{lemma}
\begin{proof}
First we recall the relations (\ref{Bind}). Let us suppose
$$
\{ i_2(\ell) + 2 \ell,   i_2(\ell) + 2 \ell + 1  \}
\bigcap
\mathfrak{R}  [i_1(\ell)] \neq \emptyset.
$$
Then there exists $m \in \mathcal{Z}$ such that
$$
i_2(\ell) + 2 \ell + \epsilon \equiv i_1(\ell) + 2m + \epsilon'  \pmod{n}
$$
for some $\epsilon, \epsilon' \in \{0, 1\}$. Since the congruence is equivalent to
$$
 \frac{n + 1 - \epsilon(n) }{2} \equiv    2 (\ell - m) + \epsilon  - \epsilon'   \pmod{n},
$$
we reach contradiction by the definition of $\mathcal{Z}$.

Similarly, let us suppose
$$
\{ i_1(\ell) + 2 \ell ,   i_1(\ell) + 2 \ell + 1 \}
\bigcap
\mathfrak{R}  [ i_2(\ell) ] \neq  \emptyset.
$$
Then there exists $m \in \mathcal{Z}$ such that
$$
i_1(\ell) + 2 \ell + \epsilon \equiv i_2(\ell) + 2m + \epsilon'  \pmod{n}
$$
for some $\epsilon, \epsilon' \in \{0, 1\}$. From this congruence, we may reach contradiction in the same way as above.
\end{proof}

\begin{lemma}
We have
$$
\mathfrak{B}_{\ell} \cap \mathfrak{B}_{\ell'} = \emptyset
$$
for any distinct $\ell, \ell' \in \mathcal{Z}$.
\end{lemma}
\begin{proof}
Let $Z = \# \mathcal{Z}$ and denote
$$
\mathcal{Z} = \{ \ell_1 , \ldots, \ell_Z  \}.
$$
We shall prove by induction that
$$
\mathfrak{B}_{\ell_i} \bigcap
\left( \mathfrak{B}_{\ell_1} \cup \cdots \cup \mathfrak{B}_{\ell_{i - 1}} \cup \widetilde{\mathfrak{B}}_{\ell_{i + 1} } \cup \cdots \cup  \widetilde{\mathfrak{B}}_{\ell_Z} \right)
= \emptyset
$$
for each $1 \leq i  \leq Z$.
The base case $i = 1$ follows easily from (\ref{pdis}) and Lemma \ref{disj}.
Let us suppose the statement holds for all values greater than or equal to $1$ and less than $i$, for some $1 < i \leq Z$.
Let $i + 1 < j \leq Z$. Then it follows from (\ref{pdis}) and Lemma \ref{disj} that
$$
\mathfrak{B}_{\ell_{i + 1}} \cap \widetilde{\mathfrak{B}}_{\ell_j}
=
\left( \mathfrak{B}_{\ell_{i + 1}} \setminus \widetilde{\mathfrak{B}}_{\ell_{i + 1}} \right) \cap \widetilde{\mathfrak{B}}_{\ell_j}
\subseteq
\left( \mathfrak{B}_{\ell_{i + 1}} \setminus \widetilde{\mathfrak{B}}_{\ell_{i + 1}} \right) \cap \mathfrak{R}
= \emptyset.
$$
Therefore, it remains to prove
$$
\mathfrak{B}_{\ell_{i + 1}} \bigcap
\left( \mathfrak{B}_{\ell_1} \cup \cdots \cup \mathfrak{B}_{\ell_i} \right)
= \emptyset.
$$
Since it follows from the inductive hypothesis that
$$
\widetilde{\mathfrak{B}}_{\ell_{i + 1}} \cap \mathfrak{B}_{\ell_s}  = \emptyset \quad (1 \leq s \leq i),
$$
it suffices to prove
$$
\left( \mathfrak{B}_{\ell_{i + 1} } \setminus  \widetilde{\mathfrak{B}}_{\ell_{i + 1} } \right)  \cap \mathfrak{B}_{\ell_s}  = \emptyset \quad (1 \leq s \leq i),
$$
which in turn follows from
\begin{eqnarray}
\label{BBBB}
\mathfrak{B}_{\ell_{i + 1}}[i_1(\ell_{i + 1})] \cap \mathfrak{B}_{\ell_s} [i_1(\ell_{i + 1})] = \emptyset
\quad
\textnormal{and}
\quad
\mathfrak{B}_{\ell_{i + 1}}[i_2(\ell_{i + 1})] \cap \mathfrak{B}_{\ell_s} [i_2(\ell_{i + 1})] = \emptyset
\end{eqnarray}
for all $1 \leq s \leq i$.
Since the $i_1(\ell_{i + 1})$-th and $i_2(\ell_{i + 1})$-th block of $\widetilde{\mathfrak{B}}_{\ell_s}$ do not get modified to construct $\mathfrak{B}_{\ell_s}$,
which follows from (\ref{modiff}), we obtain
$$
\mathfrak{B}_{\ell_s}[ i_1(\ell_{i + 1})] = \widetilde{\mathfrak{B}}_{\ell_s}[ i_1(\ell_{i + 1})] \subseteq \mathfrak{R}[ i_1(\ell_{i + 1})]
$$
and
$$
\mathfrak{B}_{\ell_s}[ i_2(\ell_{i + 1})] = \widetilde{\mathfrak{B}}_{\ell_s}[ i_2(\ell_{i + 1})]  \subseteq \mathfrak{R}[ i_2(\ell_{i + 1})]
$$
for all $1 \leq s \leq i$. Therefore, (\ref{BBBB}) and consequently the result follow from Lemma \ref{disj}.
\end{proof}

Finally, we prove that $\mathfrak{B}_\ell$ is a  linearly independent set.
\begin{lemma}
Given $\ell \in \mathcal{Z} \setminus \{ \lfloor n/4 \rfloor, \lfloor n/4 \rfloor + 1 \}$, we have
$$
\dim \textnormal{Span}_{\RR} \mathfrak{B}_\ell = 2n-1.
$$
\end{lemma}
\begin{proof}
Suppose we have the following linear combination of columns in $\mathfrak{B}_\ell$:
\begin{eqnarray}
\label{veceq}
\mathbf{0} &=&
\sum_{ \substack{ 1 \leq i \leq n  \\ i \neq  i_1(\ell),  i_2(\ell)  } }  (  x_i \c_{i, i + 2 \ell } + y_i \c_{i, i + 2 \ell + 1 }  )
\\
\notag
&&
+ x_{i_1(\ell)} \c_{i_1(\ell),  i_2(\ell) + 2 \ell  } +  y_{i_1(\ell)} \c_{ i_1(\ell),
 i_2(\ell) + 2 \ell + 1   }
\\
\notag
&&
+  x_{i_2(\ell)} \c_{ i_2(\ell),  i_1(\ell) + 2 \ell    } + y_{i_2(\ell)} \c_{  i_2(\ell),  i_1(\ell) + 2 \ell + 1},
\end{eqnarray}
with $y_{n - 2 \ell - 1} = 0$ (because we have (\ref{relnn}) and $\c_{n - 2 \ell - 1, n } \notin \mathfrak{B}_\ell$).
Our goal is to show $x_i = y_i = 0$ for all $1 \leq i \leq n$.
Let us recall that $\c_{n - 2 \ell, n} \in \mathfrak{B}_\ell$ whose entries between $(n+1)$-th and $(2n-1)$-th positions are all $0$.
Furthermore, this is the only column in $\mathfrak{B}_\ell$ of the form $\c_{*, n}$, which follows from the definition of $\mathfrak{B}_\ell$ and
(\ref{relnn}). By the definition of $\c_{i,j}$,
the vector equation (\ref{veceq}) is equivalent to the following system of linear equations
\begin{eqnarray}
x_i  +  y_i &=& 0 \quad (1 \leq  i \leq  n)
\notag
\\
\notag
x_{i + 1 } + y_{i}  &=& 0 \quad  (i \neq  i_1(\ell) - 1, i_1(\ell), i_2(\ell) - 1,  i_2(\ell), n - 2 \ell - 1)
\\
\notag
x_{i_2(\ell)} + y_{ i_1(\ell) - 1  } &=& 0
\\
\notag
x_{i_2(\ell) + 1} + y_{ i_1(\ell)}   &=& 0
\\
\notag
x_{ i_1(\ell) } + y_{ i_2(\ell) - 1 }  &=& 0
\\
\notag
x_{ i_1(\ell) + 1  } + y_{i_2(\ell)} &=& 0
\\
\notag
y_{n - 2 \ell - 1} &=& 0.
\end{eqnarray}
It is clear that this system of equations is precisely $\mathbb{L}( i_1(\ell); n - 2 \ell - 1; i_2(\ell) )$.
Therefore, the result follows from Lemma \ref{First} since we have (\ref{relnn}).
%$$
%i_1(\ell) = \frac{n}{2} - \ell  <  n - 2 \ell  - 1 < i_2 (\ell) =  n - \ell.
%$$
\end{proof}

\subsection{Final bound for $T_0$}
\label{Bound for T}
Let us set
$$
\mathcal{Z} = \left\{ 1,  \ldots,    \left\lfloor  \frac{n}{4} \right\rfloor - 1 \right \}.
$$
Then we have
\begin{eqnarray}
\notag
|  2 (m - \ell) + \delta  |  \leq  2  \left\lfloor  \frac{n}{4} \right\rfloor  - 3
\leq  \frac{n - 6}{2}
\end{eqnarray}
for any $m, \ell \in \mathcal{Z}$ and $\delta \in \{-1, 0, 1\}$, which implies  that $\mathcal{Z}$ satisfies (\ref{propZ}).
Since every column in $\mathfrak{B}_\ell$  has $0$ for the last two entries, we obtain the following as an immediate consequence.
\begin{corollary}
Let $\ell \in \mathcal{Z}$ and  $(\a, \b) \in \mathcal{S}_1 \times \mathcal{S}_2$.
Then $\mathfrak{B}_\ell \cup \{ \a, \b  \}$ is a linearly independent set.
\end{corollary}

Since
$$
\#\mathcal{S}_1, \# \mathcal{S}_2 \geq  n - 1  \geq |\mathcal{Z}| =  \left\lfloor  \frac{n}{4} \right\rfloor - 1,
$$
and $\mathfrak{B}_\ell$ are pairwise disjoint, we easily obtain $\lfloor n/4 \rfloor - 1$ pairwise disjoint sets of the form
$$
\mathfrak{D}_\ell = \mathfrak{B}_\ell \cup \{ \a_\ell, \b_\ell  \},
$$
where $(\a_\ell, \b_\ell) \in \mathcal{S}_1 \times \mathcal{S}_2$.
Thus we have established
$$
T_0 \geq  \left\lfloor  \frac{n}{4} \right\rfloor - 1,
$$
where $T_0$ is from the statement of the theorem.

\bibliographystyle{rome}
\bibliography{diag}

\end{document}